\documentclass[leqno]{article}
\usepackage[latin1]{inputenc}
\usepackage{amsmath}
\usepackage{amsthm}
\usepackage{amsfonts}
\usepackage{amssymb}
\usepackage{graphicx}
\usepackage{url}
\usepackage[all]{xy}
\usepackage{mathabx,epsfig} 
\usepackage{faktor}
\usepackage{thmtools}
\usepackage{thm-restate}
\usepackage{microtype}
\usepackage{hyperref}

\newtheorem{theorem}{Theorem}
\newtheorem{lemma}{Lemma}
\newtheorem{corollary}{Corollary}
\newtheorem{proposition}{Proposition}
\newtheorem*{remark}{Remark}

\begin{document}
	
	\title{\vspace{-2.6cm}Smoothing finite-order bilipschitz homeomorphisms of 3-manifolds}
	\author{Lucien Grillet}
	\date{}
	\maketitle
	
	\begin{abstract}
		We show that, for $\varepsilon=\dfrac{1}{4000}$, any action of a finite cyclic group by $(1+\varepsilon)$-bilipschitz homeomorphisms on a closed 3-manifold is conjugated to a smooth action.
	\end{abstract}
	
	\section{Introduction}
	
	In 1939, P. A. Smith proved \cite[7.3 Theorem 4]{Sm1939} that the fixed set $M^\sigma$ of a finite-order and orientation-preserving homeomorphism $\sigma$ of the 3-sphere $S^3$ was empty or a circle $S^1$. He then asked \cite[Problem 36]{Ei1949} if this circle could be knotted. This question is known as the Smith conjecture. More generally, the question is whether such a map is conjugated to an orthogonal map. Since the work on geometrization by Thurston and Perelman, we know that this is the case for smooth maps \cite{BLP2005}. On the other hand, Bing \cite{Bi1952} gave an example of a continuous orientation-reversing involution with a wildly embedded 2-sphere as fixed set. Therefore, this involution could not be conjugated to an orthogonal map. Montgomery and Zippin \cite{MZ1954} also modified Bing's example to obtain an orientation-preserving involution with a wild circle as a fixed set. Jani Onninen and Pekka Pankka showed in 2019 \cite{OnPa2019} that there also exists wild involutions in the Sobolev class $W^{1,p}$.

	One can then wonder what happens for maps with more regularity but which are not differentiable. In \cite{Ha2008}, D. H. Hamilton announced that quasi-conformal reflections are tame, but the proof seems to remain unpublished. In fact, even the Lipschitz case seems to be  considered open. For example, as recently as 2013, Michael Freedman asked in \cite[Conjecture 3.21]{Fr2013} if the Bing involution could be conjugated to a Lipschitz homeomorphism. Jani Onninen and Pekka Pankka reiterate this question in 2019 \cite{OnPa2019}. In this paper, we give a partial answer to Freedman's question, proving that for $\varepsilon>0$ small enough, such wild finite-order maps can not be $(1+\varepsilon)$-bilipschitz. More precisely, we will show the following theorem.
	
	\begin{restatable}{theorem}{finaltheorem}\label{finaltheorem}
		For $\varepsilon=\dfrac{1}{4000}$, any action of a finite cyclic group by $(1+\varepsilon)$-bilipschitz homeomorphisms on a closed 3-manifold is conjugated to a smooth action.
	\end{restatable}

	This theorem also implies that every $C^1$ action of a finite cyclic group on a closed 3-manifold $M$ is conjugated to a smooth action, without any condition on the norm of the derivatives of the elements of the group. Indeed, we can define a $C^0$ metric on the manifold by averaging the pullbacks of the starting metric by every element of the group. This $C^0$ metric can then be approximated as closed as desired by a smooth metric. The action will then be $(1+\varepsilon)$-bilipschitz for this last metric, which is conjugated to a smooth action by Theorem \ref{finaltheorem}. In particular this implies the following corollary.

	\begin{corollary}
		The quotient of a compact 3-manifold by a finite cyclic $C^1$ action is an orbifold.
	\end{corollary}

	This partially answers a question asked by Juan Souto in \cite{So2010}. The quotient of a compact 3-manifold by a finite and smooth action is an orbifold, but it is not easy to determine if a $C^1$ action suffices. He explains that he had to write his Theorem 0.1 for smooth actions instead of $C^1$ actions because of this problem.\\

	Theorem \ref{finaltheorem} is proved by studying the tameness of the fixed set of the action. We say that such a set is tamely embedded if there is an ambient homeomorphism sending it to a polyhedron. If the fixed sets of the elements of such an action are tamely embedded, it is known that this action is smoothable (see Theorem \ref{Kwasik}).
	
	Wes say that a set which is not tamely embedded is wildly embedded. A well known example of such a wild embedding is the Alexander horned sphere, which is a 2-sphere in $S^3$ with a complement which is not simply connected (called a Alexander horned ball). This wild sphere is depicted in Figure \ref{alexandersphere}, on which the exterior is an Alexander horned ball. A sphere with such a complement cannot be tamely embedded. The Bing involution \cite{Bi1952} is constructed using the Alexander horned ball. Bing proved namely that gluing two Alexander horned balls along their boundaries yields a 3-sphere. Exchanging these two copies produces an involution of the 3-sphere with a wildly embedded 2-sphere as a fixed set. The fact that the fixed set of the Bing involution is wildly embedded shows that it is not conjugate to a smooth action.
	
	\begin{figure}[h]
		\caption{An Alexander horned sphere}
		\centering
		\includegraphics[height=3cm]{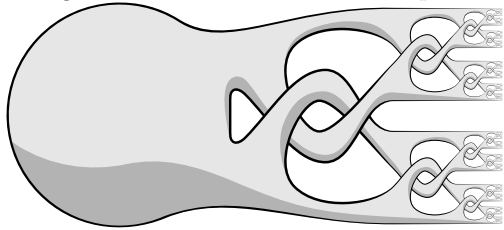}
		\label{alexandersphere}
	\end{figure}

	More precisely, the wildness of the Alexander horned ball appears in the fact that its interior is not homeomorphic to the interior of a compact manifold with boundary.

	In unpublished work, Pekka Pankka and Juan Souto already showed that the Bing involution was not conjugated to a $(1+\varepsilon)$-bilipschitz involution for every $\varepsilon>0$. They showed that if it was the case, the fixed set of this involution should have a complement homeomorphic to the interior of a compact manifold with boundary. This is however a weaker result than Theorem \ref{finaltheorem}. Indeed, there exist wild embeddings with complements homeomorphic to the interior of a compact manifold with boundary. For example, the Fox-Artin sphere constructed in \cite[Example 3.2]{FoAr1948} and depicted in Figure \ref{foxartinsphere} is a wildly embedded sphere whose complement is the disjoint union of two balls.
	
	\begin{figure}[h]
		\caption{A Fox-Artin sphere}
		\centering
		\includegraphics[height=3cm]{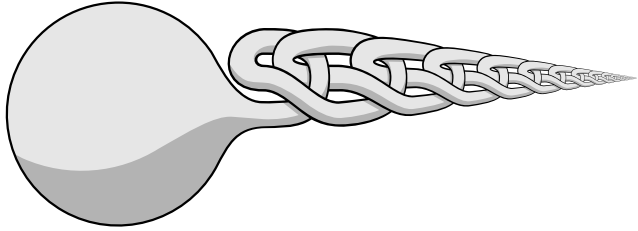}
		\label{foxartinsphere}
	\end{figure}

	We do not know if there exists an involution of $S^3$ whose fixed set is wild and has a complement homeomorphic to a disjoint union of balls, but such examples exist in dimension 4. Indeed Bing constructed \cite{Bi1957} a non-manifold $D$ whose product with $\mathbb{R}$ is homeomorphic to $\mathbb{R}^4$. The involution $x\mapsto -x$ along the $\mathbb{R}$ factor yields an involution of $\mathbb{R}^4$ whose fixed set is a non-manifold and with two balls as its complement.\\

	We will now briefly present the structure of this paper. In the first part, we recall standard results of Smith theory and we present some definitions and tools about topological tameness. We also explain how Theorem \ref{finaltheorem} can be reduced to two propositions. 
	
	The first proposition, proved in Section \ref{parttame}, is a tameness criterion allowing us to obtain, in some specific cases, the tameness of an embedding from the topology of its complement.
	
	\begin{restatable}{proposition}{tamenesscriterion}\label{tamenesscriterion}
		Let $\Sigma$ be a closed topological submanifold of a closed 3-manifold $M$. Suppose that its complement $M\setminus\Sigma$ is homeomorphic to the interior of a compact manifold $X$ with boundary. If the inclusion $i:M\setminus\Sigma\rightarrow M$ extends to a continuous map from $X$ to $M$
		
		\centerline{
			\xymatrix{
				X \ar@{.>}[rd]\\
				M\setminus\Sigma \ar@{>}[u] \ar@{->}[r]^i & M
			}
		}
		then $\Sigma$ is tamely embedded in $M$.
	\end{restatable}

	We recall that a topological submanifold is a subset which is a topological manifold for the induced topology. For example, the Alexander horned sphere is a topological submanifold of $\mathbb{R}^3$. Note that the extension off the map $i$ just needs to be continuous, without any injectivity hypothesis. This proposition will be proved by showing that the resulting map from $\partial X$ to $\Sigma$ is approximable by coverings.
	
	In the last part, we show how this criterion can be applied to finite order $(1+\varepsilon)$-bilipschitz homeomorphisms by proving Proposition \ref{lipschitzproposition}.

	\begin{restatable}{proposition}{lipschitzproposition}\label{lipschitzproposition}
		For $\varepsilon=\dfrac{1}{4000}$ and for every $(1+\varepsilon)$-bilipschitz action of a finite group $G$ on a compact Riemannian manifold $M$, the fixed set $M^G$ satisfies the conditions of Proposition \ref{tamenesscriterion}.
	\end{restatable}

	This proposition is proved by defining a Lipschitz vector field on $M$ and by showing that the flow of this vector field converges to the fixed set. This allows us to define a product structure on a neighborhood of the fixed set which extends continuously to the latter. The Lipschitz continuity of the action is crucial to define this vector field, and the bound on $\varepsilon$ is used to show the convergence of its flow.
	
	Note that Proposition \ref{lipschitzproposition} works in any dimension with any finite group. The conditions of Theorem \ref{finaltheorem} are imposed by Theorem \ref{Kwasik} and Proposition \ref{tamenesscriterion}.
		
	\begin{remark}
		As Proposition \ref{lipschitzproposition} only uses local arguments, the compactness hypothesis in its statement could be removed with some work. However, we cannot get rid of compactness needed in Theorem \ref{finaltheorem} as it is needed for Theorem \ref{Kwasik} stated in the next part.
	\end{remark}	
		
	\subsection{Acknowledgements}
	
	This work owes a lot to Juan Souto for his significant help and for his unpublished work with Pekka Pankka, which has been an encouraging starting point for this paper.
	
	\section{Proof of Theorem \ref{finaltheorem}}
	
	In this section, we recall some known results and we show how Theorem \ref{finaltheorem} can be reduced to Proposition \ref{tamenesscriterion} and Proposition \ref{lipschitzproposition}.
	
	\subsection{Smith theory}
	
	One of the earliest results in the study of finite-order homeomorphisms of $S^3$ is the determination of the topology of the fixed set by P. A. Smith. He showed (\cite[7.3 Theorem 4]{Sm1939}) that the fixed set of a finite-order homeomorphism of $S^3$ is homeomorphic to a lower dimensional sphere (that is, $S^3$ itself, a 2-sphere, a circle, a pair of points or no fixed points at all). This result can be generalized to any 3-manifold as follows.
	
	\begin{theorem}\label{smiththeory}
		Let $\sigma:M\to M$ be a finite order homeomorphism of a topological 3-manifold $M$.
		The fixed set $M^\sigma$ is a disjoint union of open pieces $(\Sigma_i)_i$ where $\Sigma_i$ is a topological manifold.
	\end{theorem}

	A proof of this result can be found, for example, in \cite[Theorem 4.5]{Pa2020}). As in John Pardon's paper, Smith theory is usually stated for prime orders, but Theorem \ref{smiththeory} is still valid for any finite order. Indeed, suppose that we have Theorem \ref{smiththeory} for prime orders, and let $\sigma$ be an homeomorphism of a topological 3-manifold of order $n$, and let $p_1\cdot ...\cdot p_k$ be a decomposition of $n$ into prime factors. The map $\sigma^\frac{n}{p_1}$ is of order $p_1$, so the result of Theorem \ref{smiththeory} applies to the fixed set $M^{\sigma^\frac{n}{p_1}}$.  The map $\sigma^\frac{n}{p_1\cdot p_2}$ is of order $p_2$ on $M^{\sigma^\frac{n}{p_1}}$ and we have the inclusion $M^{\sigma^\frac{n}{p_1\cdot p_2}}\subset M^{\sigma^\frac{n}{p_1}}$, and as Theorem \ref{smiththeory} is also true for dimensions 2 and 1, the result still applies. We can continue this argument until it applies to $\sigma$.
	
	\subsection{Tame embeddings}\label{sectiontameembeddings}
	
	Smith theory describes the topology of the fixed set, but it does not describe how the latter is embedded. And without more restrictions, wild behaviors like the Bing involution can appear.
	
	Let $\Sigma$ be a subset of a topological manifold $M$. If $\Sigma$ is a topological manifold for the topology induced by $M$, we say that $\Sigma$ is a \emph{topological submanifold} of $M$. We say that $\Sigma$ is \emph{tamely embedded} (or simply that $\Sigma$ is tame) if there is a self-homeomorphism of $M$ sending $\Sigma$ to a polyhedron. If there is no such homeomorphism, we say that $\Sigma$ is \emph{wildly embedded} (or simply that $\Sigma$ is wild).
	
	To check if a topological submanifold is tame, one can begin by studying its complement. Indeed, the complement of a tame topological submanifold is homeomorphic to the interior of a compact manifold with boundary. In particular, its fundamental group must be finitely generated. Such a criterion is sufficient to rule out examples like the Alexander horned sphere but is too weak for examples such as the Fox-Artin sphere. 
	
	In dimension 3, a criterion to show that a topological submanifold is tamely embedded is to check if it has a tubular neighborhood. We say that a topological $n$-submanifold $\Sigma$ of a $m$-manifold $M$ \emph{has a tubular neighborhood} if there is a vector bundle $E$ of basis $\Sigma$ and of fibers $\mathbb{R}^{m-n}$ whose total space can be embedded as an open set in $M$ and whose null section is $\Sigma$.
	
	\begin{proposition}\label{locallyflat}
		Let $\Sigma$ be a closed topological submanifold of a 3-manifold $M$. If $\Sigma$ has a tubular neighborhood $E$, then $\Sigma$ is tamely embedded.
	\end{proposition}

	One can give a direct proof of this fact, but we prefer to sketch a quicker argument, using the uniqueness of smooth structures in dimension 3.
	
	\begin{proof}
		Let $A$ be the unit disk subbundle of $E$ and let $B=\overline{M\setminus A}$. As $A$ and $B$ have topological product structures on their boundaries, they are topological manifolds with boundaries. We choose any smooth structure on $B$ and we choose a smooth structure on $A$ that makes $\Sigma$ smooth.
		
		For these smooth structures, we can choose a diffeomorphism $g$ from $\partial B$ to $\partial A$ which is isotopic to the identity map (see \cite[Theorem 6.3]{Mu1960}). Gluing $A$ and $B$ along this diffeomorphism gives us a smooth manifold $M'$ homeomorphic to $M$. To define this homeomorphism, remark that $\partial A$ has a neighborhood in $A$ homeomorphic to $\partial A\times [0,1]$. Define a map $h$ from $M'$ homeomorphic to $M$ by sending $B$ on itself via the identity map and, in the collar $\partial A\times [0,1]$, send $(x,t)$ on $\big(g_t(x),t\big)$, where $(g_t)_t$ is an isotopy from $g$ to the identity, and using the identity map elsewhere on $A$. The topological submanifold $\Sigma$ is then a smooth submanifold of $M'$.
		
		As $M$ and $M'$ are homeomorphic and as the smooth structure on a 3-manifold is unique, there is a diffeomorphism $d$ between $M'$ and $M$. The map $d\circ h$ is thus an homeomorphism of $M$ that makes $\Sigma$ smooth. This shows that the submanifold $\Sigma$ is tame.
	\end{proof}

	The fixed set of a smooth action on a 3-manifold is a smooth submanifold. In particular, this fixed set is tamely embedded. It turns out that the tameness of the fixed set for every power of a finite order self-homeomorphism of a 3-manifold is a sufficient condition for this map to be smoothable.
	
	\begin{theorem}\cite[Corollary 2.3]{KwLe1988}\label{Kwasik}
	A topological action of a finite cyclic group $G$ on a closed 3-manifold $M$ is smoothable if and only if, for every subgroup $H$ of $G$, the fixed set $M^H$ is tame.
	\end{theorem}

	Note that the tameness of the global fixed set $M^G$ is not sufficient. For example, consider the disjoint union of three 3-spheres $S^3\bigsqcup S^3\bigsqcup S^3$. Let $r$ be a circular permutation exchanging these three copies, and let $s$ be a wild involution on each 3-sphere such that $r$ and $s$ commute. The group generated by $r$ and $s$ is cyclic of order six but its action has an empty (and thus tame) global fixed set, and this action is not smoothable.

	\subsection{Reducing Theorem \ref{finaltheorem} to two propositions}\label{finalproof}
	
	Theorem \ref{finaltheorem} can be reduced to Proposition \ref{tamenesscriterion} and Proposition \ref{lipschitzproposition} stated in the introduction.

	\begin{proof}[Proof of Theorem \ref{finaltheorem}]
		Let $G$ be a finite cyclic group acting by $(1+\varepsilon)$-bilipschitz homeomorphisms on a closed 3-manifold. Each subgroup $H$ of $G$ is cyclic and acts by $(1+\varepsilon)$-bilipschitz homeomorphisms. By Proposition \ref{lipschitzproposition} and Proposition \ref{tamenesscriterion}, the fixed sets of these subgroups are tame. Theorem \ref{Kwasik} then applies, implying that the action of $G$ is smoothable.
	\end{proof}
	
	Section \ref{parttame} will be dedicated to the proof of Proposition \ref{tamenesscriterion}, and Section \ref{partlipschitz} to Proposition \ref{lipschitzproposition}.
	
	\section{The tameness criterion}\label{parttame}
	
	This section is dedicated to the proof of our tameness criterion.

	\tamenesscriterion*

	With the notations and setting of Proposition $\ref{tamenesscriterion}$, we denote by $f$ the continuous map from $\partial X$ to $M$ defined by the extension of $i$. Remark that $f$ is a surjective map from $\partial X$ to $\Sigma$. The boundary $\partial X$ has a neighborhood in $X$ homeomorphic to $\partial X\times[0,1]$ (see \cite{Br1962}). The subset $\Sigma$ thus has a closed neighborhood $U$ in $M$ such that $U\setminus\Sigma$ is homeomorphic to $\partial X\times[0,1[$ via an homeomorphism which extends to a continuous map $\pi$ from $\partial X\times[0,1]$ to $U$.
	\begin{alignat*}{5}
		\pi:~ & \partial X\times[0,1[&& ~\longrightarrow && ~U\setminus\Sigma ~~~&&\text{is an homeomorphism}\\
		\pi:~ & \partial X\times\{1\}&& ~\longrightarrow && ~\Sigma &&\text{is the continuous map $f$}
	\end{alignat*}

	The neighborhood $U$ of $\Sigma$ is then homeomorphic to the quotient of $\partial X\times[0,1]$ by $\pi$ (that is, by gluing the points of $\partial X\times\{1\}$ having the same image by $f$).
	
	\[U\simeq\faktor{\partial X\times[0,1]}{\pi}\]
	
	The topological submanifold $\Sigma$ is a finite disjoint union of connected manifolds of possibly different dimensions. We will study these connected components separately. As the components of dimension 0 and 3 are automatically tame, the only cases of interest are the components of dimension 1 and 2.
	
	\subsection{First case : $\Sigma$ is a surface}\label{sectionsurface}
	
	In this subsection, we suppose that the topological submanifold $\Sigma$ of $M$ is a connected surface.
	
	To show that $\Sigma$ is tamely embedded, we want to show that it has a tubular neighborhood, as explained in Section \ref{sectiontameembeddings}. Morally, the map $\pi$ is not an homeomorphism due to two obstructions. The first is the fact that the map $f$ should be of degree two, as $\Sigma$ should locally have two sides. The second is that the preimages of $f$ are not necessarily discrete (for example, $f$ can shrink a disk of  $\partial X$ to a point).
	
	\begin{center}
		\includegraphics[width=\textwidth]{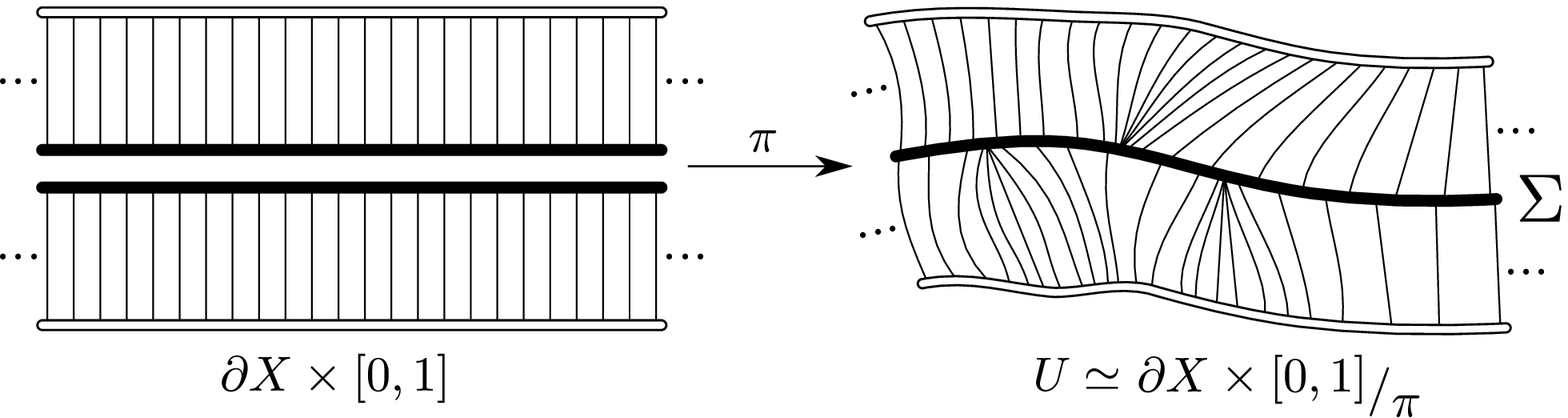}
	\end{center}

	We will deal with the first obstruction by defining a double cover $\tilde\Sigma$ of $\Sigma$ together with a new map $\tilde f$ from $\partial X$ to $\tilde\Sigma$ giving raise to a new map $\tilde\pi$ defining a new quotient space $\tilde U$.
	
	\begin{center}
		\includegraphics[width=\textwidth]{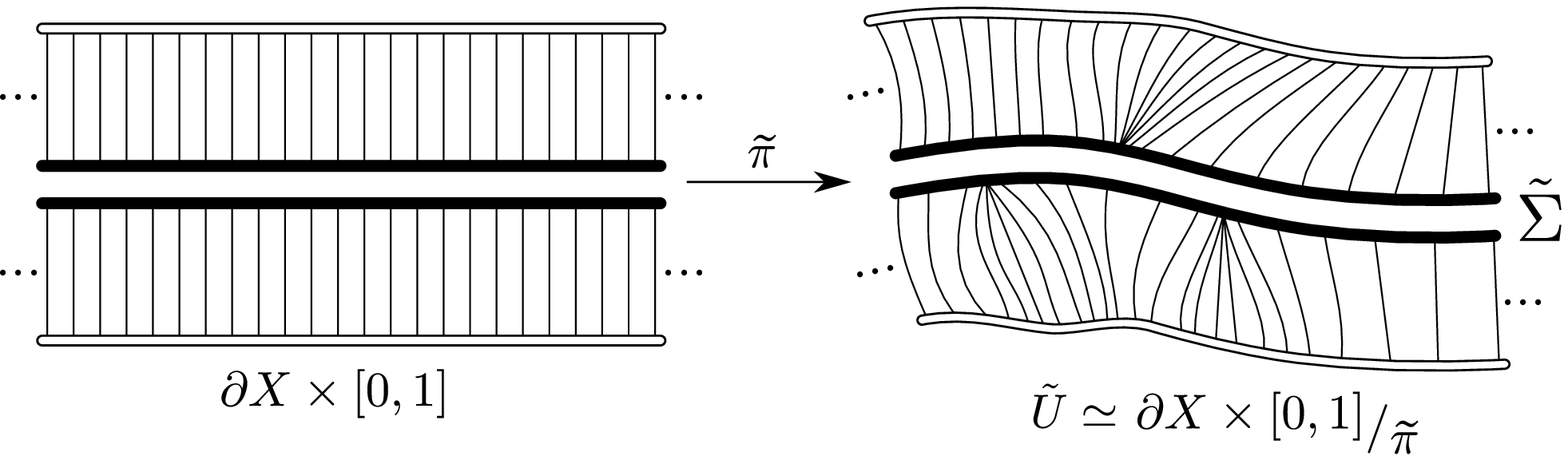}
	\end{center}

	We will deal with the second obstruction by approximating the map $\tilde f$ by homeomorphisms.

	As we do not give conditions on the orientability of $\Sigma$ and $M$, we do not know if $\Sigma$ cuts $U$ into one or two connected components. However, a local version of the Jordan-Brouwer separation theorem determines this behavior at a local sale. This theorem is a direct consequence of \cite[Theorem 2]{LE2018}.
	
	\begin{theorem}[Local Jordan-Brouwer separation theorem]\label{localjordan}
		Let $S$ be a closed topological $(n-1)$-submanifold of a $n$-manifold $M$. Then every point $x$ of $S$ has a basis of neighborhoods $(V_k^x)_k$ in $M$ such that $V_k^x\setminus S$ has exactly two connected components which both approach $x$ arbitrary closely.
	\end{theorem}
	
	We begin by unfolding $\Sigma$ to a surface $\tilde{\Sigma}$ with underlying set
	\[\tilde\Sigma=\bigcup_{x\in\Sigma}\lim_{\leftarrow}\pi_0(\mathcal{U}_x\setminus\Sigma).\]
	Here, the inverse limit is taken over the neighborhoods $\mathcal{U}_x$ of $x$ in $M$.
	
	Theorem \ref{localjordan} tells us that each point of $\Sigma$ has a basis of open neighborhoods $(V_k^x)_k$ in $M$ such that $V_k^x\setminus\Sigma$ has exactly two connected components. As each component of $V_k^x\setminus\Sigma$ approaches $x$ arbitrary closely for every $k$, the inclusion map from $\pi_0(V_k^x\setminus\Sigma)$ to $\pi_0(V_{k-1}^x\setminus\Sigma)$ is a bijection. Then, for each point $x\in\Sigma$,  $\underset{\leftarrow}{\lim}~\pi_0(\mathcal{U}_x\setminus\Sigma)$ has two elements.
	
	Let $y\in V_k^x\bigcap\Sigma$. Each element of $\underset{\leftarrow}{\lim}~\pi_0(\mathcal{U}_y\setminus\Sigma)$ naturally maps to one of the components of $V_k^x\setminus\Sigma$. The subset $\tilde V$ defined by
	\[\tilde V=\bigcup_{x\in V_k^x\bigcap \Sigma}\lim_{\leftarrow}\pi_0(\mathcal{U}_x\setminus\Sigma)\]
	separates into two subsets, depending on the component of $V_k^x\setminus\Sigma$ to which its elements map. We endow $\tilde{\Sigma}$ with the topology for which each of these subsets are open neighborhoods, for every $x$ and every $k$. This makes the projection map $p:\tilde{\Sigma}\rightarrow\Sigma$ a double covering space of $\Sigma$.
	
	For a point $y\in\partial X$, the sequence $\pi(y,1-\frac{1}{n})$ ends up in only one component of $V_k^{f(y)}$, where $\pi$ is the map defined in the beginning of Section \ref{parttame}. This defines a new continuous map $\tilde f:\partial X\mapsto\tilde\Sigma$. This allows us to define a new quotient $\tilde U$ of $\partial X\times[0,1]$ by gluing the points of $\partial X\times\{1\}$ having the same image by $\tilde f$ instead of $f$. We denote by $\tilde\pi$ the quotient map from $\partial X\times[0,1]$ to $\tilde U$.
	
	\[\tilde U = \faktor{\partial X\times[0,1]}{\tilde \pi}\]
	
	The surface $\tilde\Sigma$ will be identified with the subset $\tilde\pi(\partial X\times\{1\})$. Note that the quotient of $\tilde U$ by the covering map $p:\tilde{\Sigma}\rightarrow\Sigma$ is homeomorphic to $U$.\\
	
	The map $\tilde f$ is not an homeomorphism, but is approximable by homeomorphisms. To prove this, we will show that $\tilde f$ is \emph{cell-like}, meaning that its fibers $\tilde f^{-1}(\{x\})$ are null-homotopic in each of their neighborhoods.
	
	\begin{lemma}\label{celllike}
		The map $\tilde f$ is cell-like.
	\end{lemma}
	
	To show lemma \ref{celllike}, we begin by proving Lemma \ref{fiber} and Lemma \ref{fiberprop}.

	\begin{lemma}\label{fiber}
	 	For any point $x\in\tilde\Sigma$, the fiber $\tilde f^{-1}(\{x\})$ is connected.
	\end{lemma}

 	\begin{proof}
 		From Theorem \ref{localjordan}, every point of $\tilde\Sigma$ has a basis of closed neighborhoods $(C_k)_{k}$ in $\tilde U$ such that $C_k\setminus \tilde\Sigma$ is connected.
			
		Suppose that the fiber $\tilde \pi^{-1}(\{x\})\subset\partial X\times\{1\}$ can be written as a disjoint union of two non-empty compact sets $A_1$ and $A_2$. Choosing a distance $d$ on $\partial X\times[0,1]$, the closed set $Y$ defined by
		\[{Y=\big\{y\in \partial X\times[0,1]~|~\text{d}(y,A_1)=\text{d}(y,A_2)\big\}}\]
		is disjoint from $A_1$ and $A_2$. The sets $\tilde \pi^{-1}(C_k)$ are a basis of closed neighborhoods of $\tilde \pi^{-1}(\{x\})$, and the sets $\tilde \pi^{-1}(C_k\setminus\tilde\Sigma)$ are homeomorphic to the sets  $C_k\setminus\tilde\Sigma$ and are thus connected. The sets $\tilde \pi^{-1}(C_k)$ thus always intersect $Y$ and the intersections $\tilde \pi^{-1}(C_k)\cap Y$ are therefore closed and non-empty. Finally, the intersection $\bigcap_k \tilde \pi^{-1}(C_k)\cap Y$ is non-empty and contained in $\bigcap_k \tilde \pi^{-1}(C_k)=\tilde \pi^{-1}(\{x\})$ and in $Y$, which is impossible since $Y$ is disjoint from $A_1$ and $A_2$. The fiber $\tilde \pi^{-1}(\{x\})=\tilde f^{-1}(\{x\})$ is then connected.
	\end{proof}
	
	\begin{lemma}\label{fiberprop}
		Let $g:A\rightarrow B$ be a continuous surjective map with connected fibers between compact Hausdorff spaces. Let $B'\subset B$ and $A'=g^{-1}(B')$. If $B'$ is connected, then $A'$ is also connected. And if $B'$ is non-separating in $B$, then $A'$ is also non-separating in $A$.
	\end{lemma}
	
	\begin{proof}
		We consider that $B'$ is connected. Suppose that $A'$ is the disjoint union of two non-empty subsets $A_1$ and $A_2$ closed in $A'$. This means that $\overline{A_1}$, the closure of $A_1$ in $A$, does not intersect $A_2$ and that $\overline{A_2}$, the closure of $A_2$ in $A$, does not intersect $A_1$. As $A$ and $B$ are Hausdorff and compact, $g(\overline{A_1})$ and $g(\overline{A_2})$ are closed in $B$. As $B'$ is connected, the sets $g(\overline{A_1})\bigcap B'$ and $g(\overline{A_2})\bigcap B'$ cannot be disjoints. So, let $x$ be an element of $g(\overline{A_1})\bigcap g(\overline{A_2})\bigcap B'$. The fiber $g^{-1}(x)$ is connected and then cannot intersect both $A_1$ and $A_2$. We can consider without loss of generality that it is contained in $A_1$. However, as $x$ is in $g(\overline{A_2})$, there is an element $y$ of $\overline{A_2}$ such that $g(y)=x$. But $g^{-1}(x)$ is contained in $A_1$ and we saw that $A_1$ and $\overline{A_2}$ were disjoint. So this situation is impossible, which means that $A'$ must be connected.
		
		We now consider that $B'$ is a non-separating. This means that $B\setminus B'$ is connected. Using the first part of this lemma, we obtain that $A\setminus A'=g^{-1}(B\setminus B')$ is connected. So $A'$ is non-separating.
	\end{proof}
	
	Now, we can show that $\tilde f$ is cell-like.
	
	\begin{proof}[Proof of Lemma \ref{celllike}]
		Let $x$ be a point of $\tilde\Sigma$. Let $(U_i)_{i\in\mathbb{N}}$ be a decreasing sequence of open disks in $\tilde\Sigma$ whose intersection is $\{x\}$. We want to show that the sets $\tilde f^{-1}(U_i)$ are also homeomorphic to disks. The set $\tilde f^{-1}(\{x\})$ will then be the intersection of a decreasing family of disks, which implies that the map $\tilde f$ is cell-like.
		
		First, we know that $\tilde f^{-1}(U_i)$ is a connected surface thanks to Lemma \ref{fiberprop}. This surface has only one end. Indeed, $U_i$ is an increasing union of non-separating compacts set, and so is $\tilde f^{-1}(U_i)$ by Lemma \ref{fiberprop}.
		
		Consider the subset $\pi\big(f^{-1}\big(p(U_i)\big)\times[0,1]\big)$ of $M$, where $p:\tilde{\Sigma}\rightarrow\Sigma$ is the cover map. This neighborhood is a manifold since its boundary $\pi\big(f^{-1}\big(p(U_i)\big)\times\{0\}\big)$ has a collared neighborhood $\pi\big(f^{-1}\big(p(U_i)\big)\times[0,\varepsilon]\big)$. Endow this manifold with a smooth structure so that the map $\pi$ is smooth on this collar. Suppose that $\tilde f^{-1}(U_i)$ is not a disk, it is then homeomorphic to a non-simply connected compact surface with a point removed. We can then find two smooth curves $\gamma_1$ and $\gamma_2$ in $\tilde f^{-1}(U_i)$ with a non-zero modulo 2 intersection number. Let us define the smooth curve $\gamma'=\pi(\gamma_1\times\{\frac{\varepsilon}{3}\})$ and a smooth surface $S'$ in the following way. First, we define a topological surface $S$  by two pieces $S_1$ and $S_2$. The first piece $S_1$ is $\pi(\gamma_2\times [0,1])$. The curve $\pi(\gamma_2\times\{1\})$ is contained in $U_i$ and can be contracted into a point, as $U_i$ is a disk. This homotopy defines the second piece $S_2$. This surface $S$ defined by the two pieces $S_1$ and $S_2$ can be approximated by a smooth surface $S'$ without being modified on $\pi(\gamma_2\times[0,\frac{2\varepsilon}{3}])$. Finally, $\gamma'$ and $S'$ have an intersection number of 1, but this is not possible since $\gamma'$ can also be homotoped to $U_i$ and then contracted to a point in the same way we constructed $S$. So $\tilde f^{-1}(U_i)$ is a disk.
	\end{proof}

	We can now prove Proposition \ref{tamenesscriterion}.

	\begin{proof}[Proof of Proposition \ref{tamenesscriterion} if $\Sigma$ is a surface]
		A version of Moore's theorem \cite[25 Corollary 1A]{Da1986} states that cell-like maps between surfaces are approximable by homeomorphisms. The surface $\tilde\Sigma$ is then homeomorphic to $\partial X$ and the map $\tilde f$ is thus approximable by homeomorphisms $f_{n}$. Choose any homeomorphism $h$ from $\partial X$ to $\tilde\Sigma$ and define the following maps
		\begin{align*}
			h_{n}:\partial X&\rightarrow \partial X\\
			x&\mapsto f_{n}^{-1}\circ h (x)
		\end{align*}	
		\begin{align*}
			\varphi_{n}:\partial X&\rightarrow \tilde U\\
			x&\mapsto \tilde \pi(x,1-\dfrac{1}{n})
		\end{align*}
		The map $\varphi_{n}$ send points of $\partial X$ close to $\tilde\Sigma$, and the map $h_{n}$ is a reparametrization of $\partial X$. As the limits of $\varphi_{n}$ and of $f_{n}$ are both $\tilde f$, the limit of the sequence $\varphi_{n}\circ h_{n}=\varphi_{n}\circ f_{n}^{-1}\circ h$ is the map $h$.
		
		From here, one can apply a Bing's approximation theorem \cite[Theorem 11.1]{Bi1959} to obtain the tameness of $\Sigma$, but we will provide a more direct proof.

		For any positive real number $t\in[n,n+1]$, we can interpolate the maps $\varphi_{n}$ and $\varphi_{n+1}$ and the maps $h_{n}$ and $h_{n+1}$ to obtain maps $\varphi_{t}$ and $h_{t}$ with the same properties. To define $\varphi_{t}$, we just replace $n$ by $t$ in the definition, and $h_{t}$ is defined using local connectivity of $Homeo(\partial X)$. We can then define the following homeomorphism.
		\begin{align*}
			\Psi:\partial X\times[0,1]&\rightarrow \tilde U\\
			(x,t)&\mapsto  \left\{\begin{array}{ll}
										\varphi_{\frac{1}{1-t}}\circ h_{\frac{1}{1-t}} & \mbox{if } t<1 \\
										h(x) & \mbox{if } t=1 
								  \end{array}\right.
		\end{align*}
		For $\Psi$ to be continuous at $t=1$, the paths between the maps $h_{n}$ and $h_{n+1}$ must however be chosen carefully so that their lengths go to zero with $n$.
		
		Finally, taking the quotient of $\tilde U$ by the covering map $p:\tilde\Sigma\rightarrow\Sigma$ yields the total space of a vector bundle over $\Sigma$. As this quotient is homeomorphic to $U$, this proves that $\Sigma$ has a tubular neighborhood. By proposition $\ref{locallyflat}$, $\Sigma$ is then tamely embedded.
	\end{proof}

	\subsection{Second case : $\Sigma$ is a circle}

	In this subsection, we suppose that the subset $\Sigma$ of $M$ is a circle $S^1$. The approach will be very similar to what we did for the first case. We will show that $\partial X$ is the total space of a circle bundle over $S^1$ (that is, a torus or a Klein bottle) and that the map $\pi$ from $\partial X\times[0,1]$ to $U$ can be reparametrized to a map whose restriction to $\partial X\times\{1\}$ is the standard projection of this bundle onto its basis.
	
	In this case, there are no cell-like map from $\partial X$ to some surface, but the general method still works. Most of the lemmas used for the first case will be adapted by hand, and previously used arguments will only be sketched.
	
	\begin{lemma}
		For any point $x\in\Sigma$, the fiber $f^{-1}(x)$ is connected and non-separating.
	\end{lemma}
		
	\begin{proof}
		The proof of Lemma \ref{fiber} can be reused as it is. The important points being that the connected neighborhoods $(C_k)_{k}$ still exist and that the sets $C_k\setminus\Sigma$ are still connected, as removing a 1-dimensional topological submanifold from a 3-dimensional connected manifold cannot disconnect it.
		
		Lemma \ref{fiberprop} still shows that the fibers are also non-separating in $\partial X$.
	\end{proof}

	\begin{lemma}
		Let $I$ be an open interval of $\Sigma$. The set $f^{-1}(I)$ is homeomorphic to an open cylinder $S^1\times]0,1[$.
	\end{lemma}
	
	\begin{proof}
		We know that $\tilde f^{-1}(I)$ is a connected surface thanks to Lemma \ref{fiberprop}. This surface has two ends. Indeed, $I$ is an increasing union of compact sets $I_n$ separating $I$ into two connected components $C_n$ and $C'_n$ with $C_{n+1}\subset C_n$ and $C'_{n+1}\subset C'_n$. So, by Lemma \ref{fiberprop}, $f^{-1}(I)$ is also the increasing union of the compact sets $f^{-1}(I_n)$ separating $f^{-1}(I)$ into two connected components $f^{-1}(C_n)$ and $f^{-1}(C'_n)$ with $f^{-1}(C_{n+1})\subset f^{-1}(C_n)$ and $f^{-1}(C'_{n+1})\subset f^{-1}(C'_n)$.
		
		Suppose that $f^{-1}(I)$ is not an open cylinder, it is then homeomorphic to a non-simply connected compact surface with two points removed. We can then find two smooth curves $\gamma_1$ and $\gamma_2$ in $\tilde f^{-1}(I)$ with a non-zero modulo 2 intersection number. The argument used in Lemma \ref{celllike} then work in the same way.
	\end{proof}

	\begin{lemma}
		The surface $\partial X$ is homeomorphic to the total space of a circle bundle over $\Sigma$.
	\end{lemma}

	\begin{proof}
		Write $\Sigma$ as the union of two open intervals $I_1$ and $I_2$. These two intervals intersect in two other intervals $I_{3}$ and $I_{4}$. Let $S_3$ and $S_4$ be two non-null-homotopic circles in the cylinders $f^{-1}(I_3)$ and $f^{-1}(I_4)$. The circles $S_3$ and $S_4$ are also non-null-homotopic in $f^{-1}(I_1)$, so the compact set between $S_3$ and $S_4$ in $f^{-1}(I_1)$ is then homeomorphic to a closed cylinder $S^1\times[0,1]$. The same is true for $f^{-1}(I_2)$, so $\partial X$ is the union of two closed cylinder glued along their boundaries. So $\partial X$ is homeomorphic to the total space of a circle bundle over $\Sigma$.
	\end{proof}

	We can now conclude the proof of Proposition \ref{tamenesscriterion}.
	
	\begin{proof}[Proof of Proposition \ref{tamenesscriterion} if $\Sigma$ is a circle]
		Let $h$ be a map from $\partial X$ to $\Sigma$ so that the triple $(\partial X, \Sigma, h)$ is a circle bundle. For $n\in\mathbb{Z}$, let $f_n$ be a smooth $\frac{1}{n}$-approximation of $f$. By Sard's theorem, there is an homeomorphism $j$ between $\faktor{\mathbb{R}}{\mathbb{Z}}$ and $\Sigma$ such that each $q\in\faktor{\mathbb{Q}}{\mathbb{Z}}\subset j^{-1}(\Sigma)$ is a regular value for every $f_n$. Each preimage $f_n^{-1}\Big(j\big(\dfrac{i}{2^n}\big)\Big)$ for $0\leqslant i < 2^n$ then contains a non-null-homotopic smooth circle $S_{i,n}$. We then define an homeomorphism $h_n$ from $\partial X$ to itself that sends each circle $h^{-1}\Big(j\big(\dfrac{i}{2^n}\big)\Big)$ on $S_{i,n}$. And we also define the following map.
	\begin{align*}
		\varphi_n:\partial X&\rightarrow M\\
		x&\mapsto \pi(x,1-\frac{1}{n})
	\end{align*}
	On a circle $h^{-1}\Big(j\big(\dfrac{i}{2^n}\big)\Big)$, the limit of the sequence $(\varphi_n\circ h_n)_{n\in\mathbb{N}}$ is $j\big(\dfrac{i}{2^n}\big)$. So this sequence converges to the map $h$.
	
	For any positive real number $t$, we can interpolate the maps $\varphi_n$ and $h_n$ to obtain maps $\varphi_t$ and $h_t$ as in section \ref{sectionsurface}. We can then define the following map.
	\begin{align*}
		\Psi:\partial X\times[0,1]&\rightarrow M\\
		(x,t)&\mapsto  \left\{\begin{array}{ll}
									\varphi_t\circ h_t(x)& \mbox{if } t<1 \\
									h(x) & \mbox{if } t=1 
								\end{array}\right.
	\end{align*}
	By identifying the points having the same image by $\Psi$, we obtain the total space of a disk bundle over $\Sigma$ that maps homeomorphically to a neighborhood of $\Sigma$ and whose core is mapped on $\Sigma$. By proposition $\ref{locallyflat}$, $\Sigma$ is tamely embedded.
	\end{proof}

	Now that we have proved Proposition \ref{tamenesscriterion} in every case, we will use this tameness criterion to prove Proposition \ref{lipschitzproposition}.

	\section{Finite (1+$\varepsilon$)-bilipschitz actions}\label{partlipschitz}
	
	In this section, we show that we can use the tameness criterion developed in the previous section on finite order (1+$\varepsilon$)-bilipschitz mappings. More precisely, we show the following.
	
	\lipschitzproposition*
	
	Remark that this proposition works for any dimension and any finite group.
	
	To obtain Proposition \ref{lipschitzproposition}, we will define a continuous flow $\varphi$ near the fixed set $M^G$ such that \[\lim_{t\rightarrow\infty}\varphi_t(x)\in M^G\]
	for any point $x$ close enough to $M^G$, and such that this convergence is uniform in $x$. We will show that this flow allows us to define a compactification of the complement of $M^G$ which extends continuously (but not homeomorphically) to $M^G$.

	This flow $\varphi$ will be constructed as the flow of a Lipschitz vector field $\pmb{v}$. As presented in \cite{CB2008}, Lipschitz continuity on the parameters of an ODE is enough to obtain a continuous dependency on the initial conditions and to produce a continuous flow.
	
	\subsection{A vector field near the fixed set}\label{setting}
	
	Let $\varepsilon=\dfrac{1}{4000}$ and let $G$ and $M$ be as in Proposition \ref{lipschitzproposition}. If we were given a point $x$ near the fixed set $M^G$, then a natural place where we could look for a fixed point would be near the center of mass $B(x)$ of the orbit of $x$. We will build a vector field $\pmb{v}$ pointing towards this center of mass. In what follows, we will freely use some classic notions and facts of Riemannian geometry, such as the injectivity radius and the center of mass. We refer to \cite{Be2003} to learn more about these topics.
	
	Let $r$ be the convexity radius of $M$. The center of mass of a finite set $S$ of points in a ball of radius $r$ is the unique minimum of the function
	\[z\mapsto\sum_{s \in S}d(z,s)^2 \]
	Equivalently, it is the unique zero of the vector field
	\[z\mapsto\sum_{s\in S}\exp_z^{-1}(s).\]
	Note that, for any point $x$ at a distance $\dfrac{r}{1+\varepsilon}$ from a fixed point, the orbit $G_x$ of $x$ is included in a ball of radius $r$ around the latter. This means that $G_x$ has a well-defined center of mass for any such $x$.

	Let us proceed with some definitions.
	
	\begin{itemize}
		\item If $d(x,M^G)<\dfrac{r}{1+\varepsilon}$, let $B(x)$ be the center of mass of the orbit of $x$.
	\end{itemize}
	As explained in Section \ref{euclideansection}, $B$ has a Lipschitz dependency on $x$. This means that $\exp_x^{-1}\big(B(x)\big)$ is Lipschitz in $x$.
	\begin{itemize}
		\item Let $\pmb{v}$ be a Lipschitz vector field on $M$ with $\pmb{v}(x)=\exp_x^{-1}\big(B(x)\big)$ wherever $d(x,M^G)<\dfrac{r}{1+\varepsilon}$.
		\item Let $\varphi$ be the flow of $\pmb{v}$.
	\end{itemize}

	As explained in \cite{CB2008}, the flow $\varphi$ is well-defined and acts by homeomorphisms.
	
	To prove Proposition \ref{lipschitzproposition}, we will make use of the following lemma.
	
	\begin{lemma}\label{control}
		There are positive constants $\tau>0$ and $k<1$ and an open neighborhood $V$ of $M^G$ such that we have the inequality
		\[\|\pmb{v}\big(\varphi_\tau(x)\big)\|\leq k\|\pmb{v}(x)\|\]
		for all $x$ in $V$.
	\end{lemma}
	
	We will also need the following result, asserting that, as long as $x$ is close to $M^G$, then the point $B(x)$ is almost fixed by $G$.
	
	\begin{lemma}\label{barycenterlemma}
		There is a constant $R>0$ depending only on $M$ and an open neighborhood $V$ of $M^G$ such that, for $g_0\in G$, we have
		\[d\big(B(x),g_0B(x)\big)\leqslant R~d\big(x,B(x)\big).\]
		for all $x$ in $V$.
	\end{lemma}
	
	Now, we show how Lemma \ref{control} and Lemma \ref{barycenterlemma} can be used to prove Proposition \ref{lipschitzproposition}.
	
	\begin{proof}[Proof of Proposition \ref{lipschitzproposition}]
		We begin by defining a map $l$ that measures the length of the flow line from a point $x$ to the fixed set :
		\[l(x)=\int_0^\infty \|\pmb{v}\big(\varphi_{t}(x)\big)\|\text{d}t \]
		
		We claim that this quantity is finite and depends continuously on $x$. Indeed, from Lemma \ref{control}, when $x$ is sufficiently close to the fixed set, we have the two following inequalities :
		\begin{align*}
			\|\pmb{v}\big(\varphi_{t+\tau}(x)\big)\|&\leq k~\|\pmb{v}\big(\varphi_t(x)\big)\|\\
			\|\pmb{v}\big(\varphi_{t}(x)\big)\|&\leq 1
		\end{align*}
		
		From which we obtain :
		\[\|\pmb{v}\big(\varphi_{t}(x)\big)\|\leq \|\pmb{v}(x)\|~k^{-\frac{t}{\tau}-1}\leq k^{-\frac{t}{\tau}-1}\tag{$\star$}\label{uniform}\]
		
		We see that the he map $t\mapsto\|\pmb{v}\big(\varphi_{t}(x)\big)\|$ is bounded by an integrable map independent of $x$. Thanks to the dominated convergence theorem, this uniform integrable bound shows that $l(x)$ is finite and that the map $l$ inherits the continuity of its integrand.
		
		As $\|\pmb{v}(x)\|$ converges to $0$ as $x$ approaches $M^G$, there is a neighborhood of $M^G$ such that every flow line starting in it will stay in $V$ and then converge. We still need to show that such limits are fixed points.
		
		Using the inequality of Lemma \ref{control}, we see that the vector field $\pmb{v}$ vanishes at the limits of these flow lines. From the expression of $\pmb{v}$, these limits are thus points fixed by the map $B$. Using Lemma \ref{barycenterlemma}, we see that a point fixed by $B$ is also fixed by the action of $G$. Indeed, if a point $x_0$ is fixed by $B$ and if $g_0\in G$, we have
		\begin{alignat*}{4}
			&d\big(x_0,g_0x_0\big)&&\leqslant d\big(x_0,B(x_0)\big)&&+d\big(B(x_0),g_0B(x_0)\big)&&+d\big(g_0B(x_0),g_0x_0\big)\\
			&~&&\leqslant 0&&+R\times 0&&+(1+\varepsilon)\times 0
		\end{alignat*}
		
		To show that the complement $M\setminus M^G$ of the fixed set is homeomorphic to the interior of a compact manifold with boundary, we will define a codimension 1 topological submanifold $Z$ intersecting every flow line only once. This will show that the end of $M\setminus M^G$ is homeomorphic to $Z\times [0,1[$, proving that $M\setminus M^G$ is homeomorphic to the interior of a compact manifold with boundary.
		
		Choose $b>0$ so that every point at a distance at most $b$ of $M^G$ is in $V$ and define the subset $Z$ of $M$ as follows.
		\[Z=l^{-1}(b)\]
		
		When $t$ increases, the length $l\big(\varphi_t(x)\big)$ decreases, the set $Z$ thus intersects each flow line only once.
		
		This set is a closed topological submanifold. Indeed, as the map $l$ is continuous ans as $M^G$ is compact, the set $Z$ is also compact. Moreover, from a Lipschitz version of the flow-box theorem \cite[Theorem 4]{CB2008}, we see that every point of $Z$ has a neighborhood homeomorphic to $\mathbb{R}^{n-1}$.
		
		The map $\varphi:Z\times[0,1[\rightarrow M:(x,t)\mapsto\varphi_\frac{t}{1-t}(x)$ is an homeomorphism onto its image which describes the topology of the end of $M\setminus M^G$ and which can be compactified by adding a boundary $Z\times\{1\}$. This is done by taking the limit of the flow.
		\begin{align*}
			Z\times\{1\} & \longrightarrow ~M^G\\
			(x,1) & \longmapsto \lim_{t\rightarrow\infty}\varphi_t(x)
		\end{align*}
			
		The uniform bound (\ref{uniform}) also shows that the limit $\underset{t\rightarrow\infty}{\lim}\varphi_t$ is uniform, making this map continuous.
		
		This proves that the inclusion of $M\setminus M^G$ in $M$ extends to a  continuous map from $X$ to $M$. 
	\end{proof}
	
	Our goal is now to show Lemma \ref{control} and Lemma \ref{barycenterlemma}.
		
	\subsection{Proof of Lemma \ref{control} and Lemma \ref{barycenterlemma} in a flat setting}\label{sectionflat}
	
	To make the computations easier, we will begin by working in a flat geometry. Namely, for a fixed point $p\in M^G$, we will consider that $p$ has a sufficiently large neighborhood isometric to an open subset of $\mathbb{R}^n$. As we will only work locally, we will be able to reduce to this case. We will discuss this question in section \ref{euclideansection}.
	
	We will show the following.
	
	\begin{lemma}\label{flatcontrol}
		If $M$ is flat around a point $p$, there are positive constants $\tau>0$ and $k'<1$ and an open neighborhood $V_p$ of $p$ such that we have the inequality
		\[\|\pmb{v}\big(\varphi_\tau(x)\big)\|\leq k'\|\pmb{v}(x)\|\]
		for all $x$ in $V_p$.
	\end{lemma}

	We suppose that, for $x$ sufficiently close to $p$, there is a flat ball $V_p$ centered at $p$ and of radius $2d(x,p)$. For some coordinate system on this ball, the map $B$ and the vector field $\pmb{v}$ have simple expressions:
	\begin{align*}
		B(x)&=\dfrac{1}{\vert G\vert}\sum_{g\in G}g x\\
		\pmb{v}(x)&=B(x)-x
	\end{align*}
	
	The maps $B$ has a $(1+\varepsilon)$-Lipschitz dependency on $x$ and the vector field $\pmb{v}$ has a $(2+\varepsilon)$-Lipschitz dependency on $x$. Some classical computations about centers of mass lead to the following equation, for $x$ and $y$ in $\mathbb{R}^n$.	
	\[\dfrac{1}{\mid G\mid}\sum_{g\in G}d(y,gx)^2=d\big(y,B(x)\big)^2+\dfrac{1}{\mid G\mid}\sum_{g\in G}d\big(B(x)-gx\big)^2 \label{barycenter}\tag{$\ast$}\]
	We begin by showing the flat version of Lemma \ref{barycenterlemma}.

	\begin{lemma}\label{flatbarycenterlemma}
		If $M$ is flat around a point $p$, there is an open neighborhood $V_p$ of $p$ such that, for $g_0\in G$, we have
		\[d\big(B(x),g_0B(x)\big)\leqslant R~d\big(x,B(x)\big).\]
		for all $x$ in $V$, where $R=\dfrac{1}{40}$.
	\end{lemma}
	
	\begin{proof}
		From the Lipschitz continuity of the action of $g_0$, we have
		
		\[\dfrac{1}{\mid G\mid}\sum_{g\in G}d\big(g_0B(x),gx\big)^2\leqslant(1+\varepsilon)^2\dfrac{1}{\mid G\mid}\sum_{g\in G}d\big(B(x),gx\big)^2\]
		
		Evaluating (\ref{barycenter}) at $y=g_0B(x)$, we obtain
		\[d\big(g_0B(x),B(x)\big)^2=\dfrac{1}{\mid G\mid}\sum_{g\in 
			G}d\big(g_0B(x),gx\big)^2-\dfrac{1}{\mid G\mid}\sum_{g\in 
			G}d\big(B(x),gx\big)^2\]
		
		Together with the previous inequality, this leads to
		
		\[d\big(g_0B(x),B(x)\big)\leqslant \sqrt{2\varepsilon+\varepsilon^2}~\underset{g\in 
			G}{\text{max}}~d\big(B(x),gx\big)\]
		
		We would like to obtain a inequality depending only on $\varepsilon$ and on the 
		distance $d\big(x,B(x)\big)$. To do this, note that we have the inequality :
		\begin{align*}
			\underset{g\in G}{\text{max}}~d\big(B(x),gx\big) &\leq 
			(1+\varepsilon)~\underset{g\in G}{\text{max}}~d\big(gB(x),x\big)\\
			&\leq (1+\varepsilon)~\underset{g\in 
				G}{\text{max}}\Big(d\big(gB(x),B(x)\big)+d\big(B(x),g\big)\Big)\\
			&\leq (1+\varepsilon)\Big(\sqrt{2\varepsilon+\varepsilon^2}~\underset{g\in 
				G}{\text{max}}~d\big(B(x),gx\big)+d\big(x,B(x)\big)\Big)
		\end{align*}
		
		So
		$$\underset{g\in G}{\text{max}}~d\big(B(x),gx\big)\leq 
		\dfrac{1+\varepsilon}{1-(1+\varepsilon)\sqrt{2\varepsilon+\varepsilon^2}}~d\big(x,B(x)\big)$$
		
		And finally
		$$d\big(g_0B(x),B(x)\big)\leq 
		\dfrac{(1+\varepsilon)\sqrt{2\varepsilon+\varepsilon^2}} 
		{1-(1+\varepsilon)\sqrt{2\varepsilon+\varepsilon^2}}~d\big(x,B(x)\big)$$
		
		Taking $\varepsilon=\dfrac{1}{4000}$, the quantity $\dfrac{(1+\varepsilon)\sqrt{2\varepsilon+\varepsilon^2}} 
		{1-(1+\varepsilon)\sqrt{2\varepsilon+\varepsilon^2}}$ is smaller than $\dfrac{1}{40}$.
	\end{proof}

	We can now prove Lemma \ref{flatcontrol}.

	\begin{proof}[Proof of Lemma \ref{flatcontrol}]
		Let $x$ be a point of $V_p$, $\tau=\dfrac{1}{5}$, $k'=\dfrac{999}{1000}$ and $\delta=d\big(x,B(x)\big)$.
		
		\underline{Step 1 :} We prove that $d\big(\varphi_t(x),x\big)\leq\dfrac{\delta}{3}$ for every $t\leq\tau$.
		
		If $d\big(\varphi_t(x),x\big)>\dfrac{\delta}{3}$ for some $t\leq\tau$, let
		\[t_0=\min\Big\{t\leq\tau~\mid~ d\big(\varphi_t(x),x\big)\geq\dfrac{\delta}{3}\Big\}.\]
		For $t\leq t_0$, we have $d\big(\varphi_t(x),x\big)\leq\dfrac{\delta}{3}$, so $d\Big(B \big(\varphi_t(x)\big),B(x)\Big)\leq\dfrac{\delta}{3}(1+\varepsilon)$ as $B$ is $(1+\varepsilon)$-Lipschitz. As $\pmb{v}\big(\varphi_t(x)\big)=B\big(\varphi_t(x)\big)-\varphi_t(x)$, $\varphi_t(x)$ is contained in the convex hull of $\{x\}\bigcup\mathcal{B}\big(B(x),\dfrac{\delta}{3}(1+\varepsilon)\big)$ for every $t\leq t_0$ (see Figure \ref{step1}).
		
		The distance between $\varphi_t(x)$ and $B\big(\varphi_t(x)\big)$ is then smaller than the diameter of this convex hull, which is $\delta+\dfrac{\delta}{3}(1+\varepsilon)=\dfrac{\delta(4+\varepsilon)}{3}$.
		So $\|\pmb{v}\big(\varphi_t(x)\big)\|\leq \dfrac{\delta(4+\varepsilon)}{3}$, then $d\big(x,\varphi_{t_0}(x)\big)\leq t_0\dfrac{\delta(4+\varepsilon)}{3}\leq\dfrac{1}{5}\dfrac{\delta(4+\varepsilon)}{3}<\dfrac{\delta}{3}$. This is in contradiction with the definition of $t_0$, so $d\big(\varphi_t(x),x\big)$ cannot exceed $\dfrac{\delta}{3}$ for $t\leq\tau$.
		
		\begin{figure}[h]
			\caption{Setting of Step 1}
			\centering
			\includegraphics[width=\textwidth]{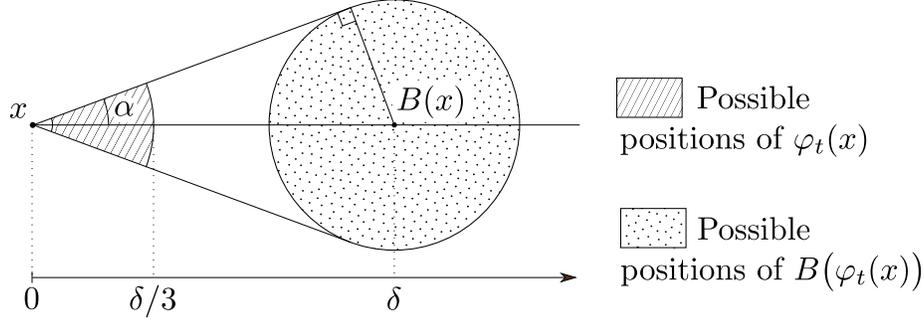}
			\label{step1}
		\end{figure}
		
		\underline{Step 2 :} We prove that $d\big(\varphi_\tau(x), B(x)\big)<\dfrac{19}{20}~\delta$.

		As $d\big(\varphi_t(x),x\big)\leq\dfrac{\delta}{3}$ for every $t\leq\tau$, we obtain $d\Big(B \big(\varphi_t(x)\big),B(x)\Big)\leq \dfrac{\delta}{3}(1+\varepsilon)$ and then $\|\pmb{v}\big(\varphi_t(x)\big)\|\geq \delta -\dfrac{\delta}{3}-\dfrac{\delta}{3}(1+\varepsilon)=\dfrac{\delta}{3}(1-\varepsilon)$.
		
		The distance between $\varphi_\tau(x)$ and $B(x)$ will be the greatest if $\pmb{v}\big(\varphi_t(x)\big)$ makes an angle $\alpha=\arcsin(\dfrac{1+\varepsilon}{3})$ with the vector $B(x)-x$ for every $t\leq\tau$ (see Figure \ref{step1}). At time $\tau$, if $\pmb{v}\big(\varphi_t(x)\big)$ made an angle $\alpha$ with  $B(x)-x$ for every $t\leq\tau$, $\varphi_\tau(x)$ would be at distance at least $\tau\dfrac{\delta}{3}(1-\varepsilon)$ from $x$, and at distance at most $\sqrt{\big(\tau\dfrac{\delta}{3}(1-\varepsilon)\big)^2+\delta^2-2\tau\dfrac{\delta^2}{3}(1-\varepsilon)\cos(\alpha)}\leq \dfrac{19}{20}\delta$ from $B(x)$.
		
		\underline{Step 3 :} We prove that $d\Big(\varphi_\tau(x), B\big(\varphi_\tau(x)\big)\Big)\leq \dfrac{999}{1000}~\delta$.
		
		The points in the orbit of $\varphi_\tau(x)$ verify the following inequality, for every $g\in G$.
		\[d\big( g\varphi_\tau(x),gB(x)\big)\leq(1+\varepsilon)~  d\big(\varphi_\tau(x),B(x)\big)\]
		Which, by Lemma \ref{flatbarycenterlemma}, leads to :
		\[d\big( g\varphi_\tau(x),B(x)\big)\leq\big(1+\varepsilon+R\big)d\big( \varphi_\tau(x)-B(x)\big)\tag{1}\]
		These points also verify :
		\[d\Big( g\varphi_\tau(x),gB\big(\varphi_\tau(x)\big)\Big)\geq\dfrac{1}{1+\varepsilon}~  d\Big(\varphi_\tau(x),B\big(\varphi_\tau(x)\big)\Big)\]
		Which, also by Lemma \ref{flatbarycenterlemma}, leads to :
		\[d\Big( g\varphi_\tau(x),B\big(\varphi_\tau(x)\big)\Big)\geq\big(\dfrac{1}{1+\varepsilon}-R\big)~  d\Big(\varphi_\tau(x),B\big(\varphi_\tau(x)\big)\Big)\tag{2}\]
		
		\begin{figure}[h]
			\caption{Limit case of Step 3}
			\centering
			\includegraphics[width=10cm]{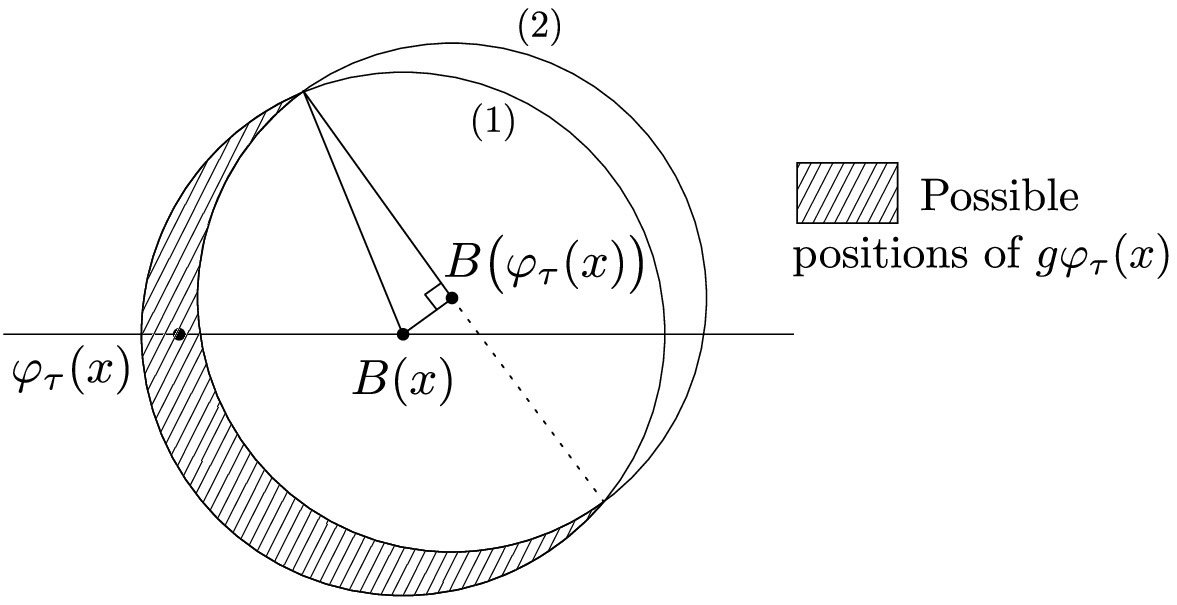}
			\label{step3}
		\end{figure}
		
		As $B\big(\varphi_\tau(x)\big)$ is the center of mass of the orbit $G\varphi_\tau(x)$, it must be contained in the convex hull of the set of points verifying inequalities $(1)$ and $(2)$. As shown in in Figure \ref{step3}, which present the limit case, the condition for $B\big(\varphi_\tau(x)\big)$ to be in this convex hull is given by 
		the inequality :
		
		\begin{align*}
		\big(1+\varepsilon+R\big)^2d\big( \varphi_\tau(x)-B(x)\big)^2\geq &~\big(\dfrac{1}{1+\varepsilon}-R\big)^2  d\Big(\varphi_\tau(x),B\big(\varphi_\tau(x)\big)\Big)^2\\
		&+d\Big( B\big(\varphi_\tau(x)\big)-B(x)\Big)^2
		\end{align*}
	
		Taking $\varepsilon=\dfrac{1}{4000}$, $R=\dfrac{1}{40}$ and $d\big(\varphi_\tau(x)-B(x)\big)=\dfrac{19}{20}\delta$, the positions of $B\big(\varphi_\tau(x)\big)$ satisfying this inequality are in the interior of an ellipse contained in the ball of center $\varphi_\tau(x)$ and of radius $\dfrac{999}{1000}~\delta$. So we necessary have $d\Big(\varphi_\tau(x), B\big(\varphi_\tau(x)\big)\Big)\leq \dfrac{999}{1000}~\delta$.
		
		So $\|\pmb{v}\big({\varphi_\tau}(x)\big)\|\leq\dfrac{999}{1000}~\delta=k'\|\pmb{v}(x)\|$
	\end{proof}
	
	\subsection{Reduction to the flat case}\label{euclideansection}
	
	In section \ref{sectionflat}, we worked locally on a flat neighborhood of a fixed point. In general, the manifold $M$ cannot be flatten in a neighborhood of $M^G$, but as every manifold is locally almost flat, Lemma \ref{flatcontrol} will allow us to produce the inequalities needed for Lemma \ref{control} and Lemma \ref{flatbarycenterlemma} will allow us to prove Lemma \ref{barycenterlemma}.
	
	\begin{proof}[Proof of Lemma \ref{control}]
		Let $p$ be a fixed point of $G$. On a small neighborhood of $p$, the manifold $M$ is almost flat. The goal of this proof is to compare the flow $\varphi$ obtained with the metric of $M$ and the flow $\varphi^E$ obtained by the flat metric of $T_pM$. Every object obtained in the flat metric of $T_pM$ will be noted with the exponent $E$.
		
		Let $U$ and $U'$ be the balls of center $p$ and of radii $\delta$ and $(1+\varepsilon)\delta$ for a $\delta$ small enough so that $U'$ is contained in $V$.

		First, notice that $B(x)$ has a Lipschitz dependency on $x$ and on the Riemannian metric $\rho$. This fact can be proved using a Lipschitz version of the implicit function theorem on the map
		\begin{align*}
			\Phi:U\times U'\times\mathcal{M} &\longrightarrow \mathbb{R}^n\\
			(x,y,\rho)&\longmapsto \rho\Big(\sum_{g\in G}\exp_y^{-1}(gx),~E_i\Big)
		\end{align*}
		where $\mathcal{M}$ is the space of all metrics on $U'$ for the uniform operator norm (according to the starting metric of $M$) and $(E_i)_i$ is a basis of sections of $TM$. The map $\Phi$ is then $K_1$-bilipschitz in $x$ and $y$ for some constant $K_1$. As the norm of $\sum\limits_{g\in G}\exp_y^{-1}(gx)$ is smaller than $K_2\delta$ for some constant $K_2$, the map $\Phi$ is $K_3\delta$-Lipschitz in $\rho$ for some constant $K_3$. The map $B$ is then $K_1^2$-Lipschitz in $x$ and $K_1K_3\delta$-Lipschitz in $\rho$.
		
		The Riemannian metric $\rho$ on $U'$ is always at a distance $K_4\delta$ from an Euclidean metric (namely, the metric of $T_pM$ induced by the exponential map), for some constant $K_4$. The ratio $\dfrac{d^E(x,y)}{d(x,y)}$ is then between $1-K_4\delta$ and $1+K_4\delta$ for any $x$ and $y$ in $U'$.
		
		With the above, the map $B_E$ is at a distance at most $K_1K_3K_4\delta^2$ from $B$. The distance between $\pmb{v}$ and $\pmb{v}^E$ is then itself bounded by $K_5\delta^2$ for some constant $K_5$. And the point $\varphi_\tau(x)$ is at a distance at most $K_6\delta^2$ from $\varphi^E_\tau(x)$ for some $K_6$.
		
		Lemma \ref{flatcontrol} gives us
		\[d^E\Big(\varphi^E_\tau(x),B^E\big(\varphi^E_\tau(x)\big)\Big)\leq k'd^E\big(x,B^E(x)\big)\]
		from which we obtain from what precedes
		\[d\Big(\varphi_\tau(x),B\big(\varphi_\tau(x)\big)\Big)\leq k'\delta+K\delta^2=(k'+K\delta)\delta\]
		for some constant $K$ depending only on $M$. Consequently, for any manifold $M$, it is always possible to work as locally as we want (i.e. to choose $\delta$ small enough) so that the inequality
		\[\|\pmb{v}^E\big(\varphi^E_\tau(x)\big)\|^E\leq k'\|\pmb{v}^E(x)\|^E\]
		implies the inequality
		\[\|\pmb{v}\big(\varphi_\tau(x)\big)\|\leq \dfrac{k'+1}{2}\|\pmb{v}(x)\|\]
		on $V$, which implies Lemma \ref{control}.
	\end{proof}
	
	We can also prove Lemma \ref{barycenterlemma} using the preceding computations.

	\begin{proof}[Proof of Lemma \ref{barycenterlemma}]
		From lemma \ref{flatbarycenterlemma}, we obtain
		\[d^E\big(B^E(x),g_0B^E(x)\big)\leqslant \dfrac{1}{40}~d^E\big(x,B^E(x)\big)\]
		for every $x$ in $V$.
		With the preceding method, we obtain
		\[d\big(B(x),g_0B(x)\big)\leqslant R~d\big(x,B(x)\big)\]
		for some constant $R>0$ depending only on $M$.
	\end{proof}
	
	\bibliographystyle{alpha}
	\bibliography{lipschitz-smith-conjecture}
	
\end{document}